\documentclass[twoside,a4paper]{amsart}
\usepackage{style}
\usepackage{mathabx}
\usepackage{tikz}
\usepackage{graphicx}
\usepackage{xcolor}
\usepackage{pdfsync}

\usepackage{xcolor}
\definecolor{webgreen}{rgb}{0,.5,0}
\definecolor{webbrown}{rgb}{.6,0,0}
\definecolor{RoyalBlue}{cmyk}{1, 0.50, 0, 0}
\usepackage[colorlinks=true, breaklinks=true, urlcolor=webbrown, linkcolor=RoyalBlue, citecolor=webgreen,backref=page]{hyperref}

\usepackage{courier} % tt
\normalfont
\usepackage{epsfig, graphicx, subfigure}
\usepackage{amsmath, amssymb}

\let\Re=\undefined
\DeclareMathOperator{\Re}{Re}
\let\Im=\undefined
\DeclareMathOperator{\Im}{Im}

\def\ge{\geqslant}
\def\le{\leqslant}

\newtheorem{theorem}{Theorem}[section]

\newtheorem{corollary}[theorem]{Corollary}
\newtheorem{lemma}[theorem]{Lemma}

\theoremstyle{remark}

\definecolor{darkbrown}{RGB}{150,1,33} 
\numberwithin{equation}{section}
\usepackage{hyperref}
\begin{document}

\title[The strong version \ldots \ldots]{
The strong version of Nonlinear Carleson Conjecture fails }

\begin{abstract} In the context of the Dirac equation  with $L^2(\R)$-potential, we study the Jost solutions and prove that the maximal function associated with the argument of the transmission coefficient is unbounded. We also prove that the strong version of the Nonlinear Carleson Conjecture (NCC) fails for Dirac equations and Krein systems.
\end{abstract} \vspace{1cm}

\author[Sergey A. Denisov]{Sergey A. Denisov}
\address{Department of Mathematics, University of Wisconsin-Madison, 480 Lincoln Dr., Madison, WI 53706, USA}
\email{\href{mailto:denissov@wisc.edu}{denissov@wisc.edu}}

\thanks{
This research was supported by the grants  NSF-DMS-2450716, Simons Fellowship in Mathematics, Simons Travel Support for Mathematicians Award, and the Van Vleck Professorship Research Award. The author gratefully acknowledges the hospitality of IHES where part of this work was done.
}



\maketitle

\setcounter{tocdepth}{3}

\section{Krein systems, maximal functions, and the NCC}

Consider the fundamental matrix $X(x,k,A)$ for the Krein system which is defined as the solution to the Cauchy problem 
\begin{equation}\label{ks1}
\partial_x{X}=\left(
\begin{array}{cc}
ik & -\overline{A}\\
-A&0
\end{array}
\right)X, \quad X(0,k,A)=\left(\begin{array}{cc}1&0\\0&1\end{array}\right)\,,
\end{equation}
where $k=\xi+i\eta\in \C$ and $A(x)\in L^1_{\rm loc}(\R^+)$.
We can write (see \cite{DKr}, p.38)
$
X=\left(\begin{smallmatrix}
\frak{A}^* & \frak{B}^*\\
\frak{B} & \frak{A}
\end{smallmatrix}\right)
$,
where $\frak A(x,k,A)$ and $\frak B(x,k,A)$, as functions in $k$, are entire functions of exponential type at most $x$ and they satisfy
\begin{eqnarray}
\frak{A} \frak{A}^*-\frak{B} \frak{B}^*=e^{ikx}, \quad k\in \C,\\
|\frak{A}|^2=1+|\frak{B}|^2, \quad k\in \R,\\
|\frak{A}|^2\ge 1+|\frak{B}|^2, |\frak{A}|^2\ge1+|\frak{B}^*|^2, \quad k\in \C^+\,.
\end{eqnarray}
Here, we use convention $f^*(k):=e^{ikx}\overline{f(\bar{k})}$ for a  function $f$ defined in $\C$ and such operation is $x$-dependent.
Given $A\in L^2(\R^+)$, we have two limits as $x\to +\infty$:
\begin{eqnarray}\label{lim1}
\frak{A}(x,k,A)\to \frak{a}(k,A),\quad
\frak{B}(x,k,A)\to \frak{b}(k,A)
\end{eqnarray}
and the convergence is locally uniform in $k\in\C^+$ (see \cite{DKr}, Section 12). In folklore, the question of existence of limits $\lim_{x\to \infty}\frak{A}(x,\xi,A)$ and $\lim_{x\to \infty}\frak{B}(x,\xi,A)$ for a.e. $\xi\in \R$  often goes  under the name {\it Nonlinear Carleson Conjecture} (NCC). See  \cite{musc,st} for related results and discussion. This problem is  motivated by, e.g., studying the stationary and non-stationary scattering in Dirac and Schr\"odinger equations. In fact, existence of such limits often implies the existence of wave operators  for the corresponding Schr\"odinger dynamics \cite{ck,dm}. The weaker version of NCC asks to study the existence of the limit $\lim_{x\to\infty} \frak{B}(x,\xi,A)/\frak{A}(x,\xi,A)$, so we will address existence of  $\lim_{x\to \infty}\frak{A}(x,\xi,A)$ and $\lim_{x\to \infty}\frak{B}(x,\xi,A)$ for a.e. $\xi\in \R$ as {\it a strong version of NCC}.  In the case when $A\in L^p(\R^+), p\in [1,2)$, the convergence and the corresponding estimates for maximal functions are known \cite{ck,st,kov1}. For $p>2$, there are examples for which the limits do not exist (check \cite{kls}, where this problem was addressed for the Schr\"odinger operators). In the current work, we will study the borderline case  $p=2$. The note \cite{Den25} (see also \cite{tt,kov}) contains discussion of NCC  in the context of polynomials orthogonal on the unit circle (OPUC) which has some advantages. We, however, prefer to address NCC for Krein systems and Dirac equations because these two models exhibit more scaling properties than the OPUC setup and that makes our analysis less technical.\medskip

\noindent{\it Notation.} The symbol $C$ denotes the absolute constant which can change the value from formula to formula. If we write, e.g., $C(\alpha)$, this defines a positive function of parameter $\alpha$. For two non-negative functions $f_1$ and $f_2$, we write $f_1\lesssim f_2$ if  there is an absolute
constant $C$ such that $f_1\le Cf_2$ for all values of the arguments of $f_1$ and $f_2$. We define $\gtrsim$
similarly and say that $f_1\sim f_2$ if $f_1\lesssim f_2$ and
$f_2\lesssim f_1$ simultaneously. If $|f_3|\lesssim f_4$, we will write $f_3=O(f_4)$. If $\alpha$ is a parameter, we write $f_1\le_\alpha f_2$ if $f_1\le C(\alpha)f_2$. If $E$ is a measurable subset of $\R$, the symbol $\chi_E$ denotes its characteristic function and $|E|$ indicates its Lebesgue measure. For the $d\times d$ matrix $A$, the symbol $\|A\|$ denotes its operator norm in $\ell^2(\C^d)$. The class $\Sch(\R)$ is the class of Schwartz functions on $\R$, $C^\infty_c(\R^+)=\{f\in C^\infty(\R^+), \supp f\subset (0,\infty)\}$ (hence $f\equiv 0$ around the origin).  \medskip

 In this note, we propose studying the following two questions that address stronger versions of NCC.

 \begin{picture}(20,180)
 \thicklines
\put(10,50){\line(1,0){300}}
\put(100,50){\line(1,1){80}}
\put(100,50){\line(-1,1){80}}
\put(100,50){\circle{2}}
\put(120,95){\circle{2}}
\put(124,95){$k$}
\put(100,35){$\xi$}
\put(80,85){$S_\xi$}
\put(200,30){Figure 1. The Stolz angle $S_\xi$.}
\end{picture}
Assume $A\in L^2(\R^+)$. \medskip

\noindent \hypertarget{id1}{{\bf {Q1}\,}}(the strong version of NCC).
 {\it Do the limits $\lim_{x\to \infty}\frak{A}(x,\xi,A)$ and $\lim_{x\to \infty}\frak{B}(x,\xi,A)$ exist for a.e. $\xi\in \R$?}

 \medskip

It is known that the function $\frak{A}$ is outer in $\C^+$ and it can be written as (\cite{DKr}, formula (12.29))
\[
\frak{A}(x,k,A)=\exp\left(\frac{1}{\pi i}\int \frac{\log|\frak{A}(x,s,A)|}{s-k}ds\right), \quad k\in \C^+
\]
so $\log \frak{A}$ is correctly defined and
\[
\arg \,\frak{A}(x,k,A)=\Im \log \frak{A}(x,k,A)=\Im \left(\frac{1}{\pi i}\int \frac{\log|\frak{A}(x,s,A)|}{s-k}ds\right)\,.
\]
The same statements hold for the limit function $\frak{a}(k,A)$.\smallskip

\noindent \hypertarget{id2}{{\bf {Q2}}\,} (the bounds for  maximal functions).  Let $S_\xi$ be a Stolz angle based on $\xi\in \R$. Define two maximal functions
\[
\frak{M}(\xi,A)=\sup_{k\in S_\xi, x\in \R^+}|\arg\, \frak{A}(x,k,A)|\,, \quad \frak{M}_w(\xi,A)=\sup_{x\in \R^+}|\arg\, \frak{A}(x,\xi,A)|,\, \quad \xi\in \R\,.
\]
Clearly, $\frak{M}_w(\xi,A)\le \frak{M}(\xi,A)$.\medskip

{\it Are  the bounds $\frak{M}(\xi,A)<\infty$ or $\frak{M}(\xi,A)<\infty$ true for a.e. $\xi\in \R$?
If so, do we have  weak-$L^1$ estimates such as
\begin{equation}\label{we}
|\{\xi: \frak{M}(\xi,A)\ge \lambda \}|\lesssim \frac{\|A\|_2}{\lambda} \quad {\it or} \quad  |\{\xi: \frak{M}_w(\xi,A)\ge \lambda \}|\lesssim \frac{\|A\|_2}{\lambda}, \quad \forall \lambda>0
\end{equation}
at least in the perturbative regime when $\|A\|_2\le \delta_0$ where $\delta_0$ is  some small positive number?}\medskip

Our main result is the following theorem.
\begin{theorem}\label{mt}Both questions Q1 and Q2 have negative answers.
\end{theorem}
\medskip

\begin{figure}[htbp]
    \centering
    % Example: trim 1cm from the bottom and 0.5cm from the top
    \includegraphics[clip, trim=0cm 12cm 0cm 0.5cm, width=\linewidth]{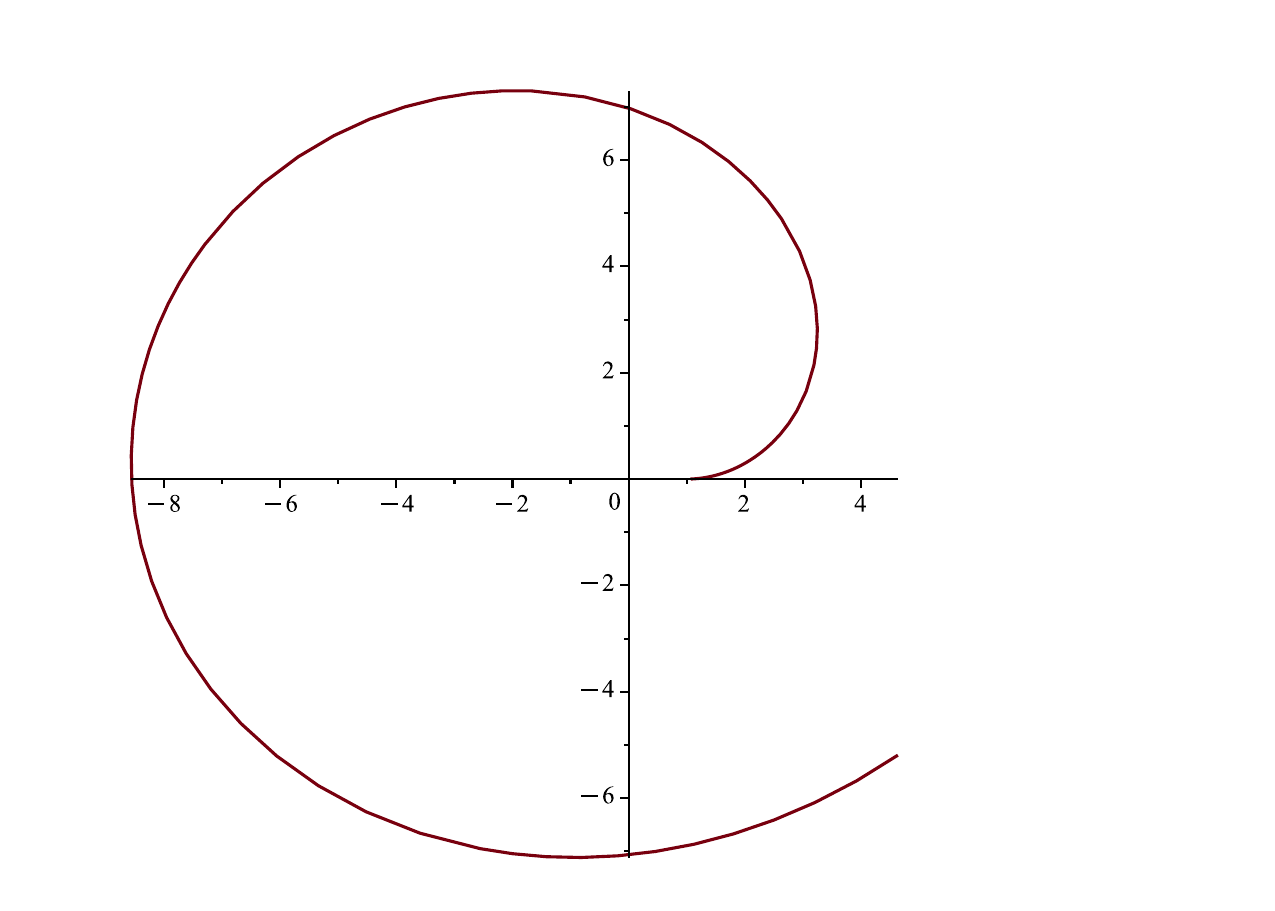}
   {Figure 2. The image of the segment $\{k=2+i\eta, \eta\in [0,100]\}$ under the $\frak{A}(7,k,A)$--map. Taking $\eta=100$ gives us a point close to $z=1$.}
    \label{pic2}
\end{figure}

\noindent{\bf Remark.} Recall that  $|\frak{A}(x,k,A)|\ge 1$ for $\Im k\ge 0$.
In this note, we will be mainly interested in the argument (rotation) of $\frak{A}(x,k,A)$.  Since $\frak{A}(x,\infty,A)=1$, it is natural to focus on $\log \frak{A}(x,k,A),\, \Im k\ge 0$ and control its imaginary part choosing the branch of log such that $\log \frak{A}(x,\infty,A)=0$. It becomes most explicit if we choose a path in $\C^+$ connecting $k=\xi\in \R$ to a point at infinity and  check the total variation of the argument over the trajectory traced by  $\frak{A}(x,k,A)$ when $k$ follows that path. For example, in the Figure~2,  we  take $A=\chi_{0\le x\le 7}$, $\xi=2$, $x=7$, compute $\frak{A}(7,k,A)$ and Maple-plot such trajectory by choosing the ray $\{k=2+i\eta, \eta\in [0,+\infty)\}$ as such path. \smallskip

We will provide the proof to Theorem \ref{mt} in the second section in the more general context of the Dirac equations, see Corollary \ref{cm}. However, the lack of \eqref{we} for the stronger version of maximal function is easy enough and follows from the following result.

\begin{theorem} \label{ks12}There is a sequence of compactly supported functions $\{\Psi_n\}$ such that $\|\Psi_n\|_2<C$ and 
\begin{equation}\label{gg1}
|\{\xi: \frak{M}(\xi,\Psi_n)\ge C'\log n \}|\gtrsim 1
\end{equation}
for some positive constants $C$ and $C'$ hence \eqref{we} fails for $\mathfrak{M}$.
\end{theorem}

\begin{proof}  We will need the following two scaling properties of the fundamental matrix $X$:\smallskip

\noindent {\it 1. Modulation.} Let $\ell\in \R$, $A^{[\ell]}:=Ae^{i\ell x}$ and
$
X^{[\ell]}:=\left(
\begin{smallmatrix}
e^{-i\ell x} &0\\
0&1\end{smallmatrix} \right)X\,.
$
Then, $\partial_xX^{[\ell]}=\left(
\begin{smallmatrix}
i(k-\ell) & -\overline{A}^{[\ell]}\\
-A^{[\ell]}&0
\end{smallmatrix}
\right)X^{[\ell]}$,  $ X^{[\ell]}(0,k)=\idm$. Hence,
\begin{equation}\label{sym-tr}
\left(
\begin{array}{cc}
e^{-i\ell x} &0\\
0&1\end{array} \right)X(x,k,A)=X(x,k-\ell,Ae^{i\ell x})
\end{equation}
and $\mathfrak{A}(x,k,A)=\mathfrak{A}(x,k-\ell,Ae^{i\ell x}),  \mathfrak{B}(x,k,A)=\mathfrak{B}(x,k-\ell,Ae^{i\ell x})$.
\smallskip

\noindent {\it 2. Dilation.} Suppose $\mu>0$, $A^{(\mu)}(x):=\mu A(\mu x)$, and $X^{(\mu)}:=X(\mu x, \mu^{-1}k, A)$. We have
\[
\partial_x X^{(\mu)}=\left(
\begin{array}{cc}
ik &-\overline{A}^{(\mu)}\\
-A^{(\mu)}& 0
\end{array}\right)X^{(\mu)}, \quad X^{(\mu)}(0,k)=\left(\begin{array}{cc}1&0\\0&1\end{array}\right)\,.
\] 
Hence, 
\begin{equation}
X(\mu x,\mu^{-1} k,A)=X(x,k,A^{(\mu)})
\end{equation}
and $\mathfrak{A}(\mu x,\mu^{-1}k,A)=\mathfrak{A}(x,k,A^{(\mu)}),  \mathfrak{B}(\mu x,\mu^{-1}k,A)=\mathfrak{B}(x,k,A^{(\mu)})$.\smallskip

Now, we continue with the proof which consists of two steps. \medskip

\noindent {\bf First step: simplified decoupled model.} Consider a function $A(x)\in C_c^\infty(\R^+)$ not identically equal to zero and supported on the interval $(0,1)$. This $A$ defines the direct scattering transforms $A(x)\stackrel{(\frak{b})}\mapsto \frak{b}(\xi,A)$ and $A(x)\stackrel{(\frak{a})}\mapsto \frak{a}(\xi,A)$, where $\frak{a}(\xi,A)$ and $\frak{b}(\xi,A)$ are the boundary values of $\frak{a}(k,A)$ and $\frak{b}(k,A)$, given in~\eqref{lim1}. Recall that $\frak{a}(k,A)$ is an outer  function in $\C^+$ and $|\frak{a}(\xi,A)|^2=1+|\frak{b}(\xi,A)|^2, \xi\in \R$.  Clearly, $\frak{a}(k,A)=\frak{A}(1,k,A)$ and $\frak{b}(k,A)=\frak{B}(1,k,A)$. Hence, $\frak{a},\frak{b}\in C^\infty(\R)$ and (we will now drop the dependence on $A$ and write $\frak{a}$ for shorthand) \[\frak{a}(k)=1+ik^{-1}\int_0^\infty |A|^2dx+O(|k|^{-2})\,,\] when $k\in \C^+$ and $|k|\to\infty$. Moreover, $\log |\frak{a}|$ is nonnegative and the {\it sum rule} holds
(see \cite{DKr}, (12.2))
 \[\int_{\R} \log |\frak{a}|d\xi=\pi\int_0^\infty |A|^2dx.\] Take $\nu  \ge 1$. The dilation argument with $\mu=\frac 1\nu  $ given above shows that
\[
A^{(1/\nu  )}(x)=A(x/\nu  )/\nu  \stackrel{(\frak{a})}{\mapsto} \frak{a}^{(1/\nu  )}(\xi):=\frak{a}(\nu  \xi)\,.
\]
Then, let 
\begin{equation}\label{lo8}
A_j(x):=A^{(1/\nu  )}(x)e^{-i\xi_jx},\quad \xi_j:=j/\nu  , \quad j\in \{0,\ldots,\nu  -1\}\,.
\end{equation}
Clearly, we have
\[
\int_0^\infty |A_j(x)|^2dx=\nu  ^{-1}\int_0^\infty |A(x)|^2dx\,.
\]
The previously discussed modulation scaling shows that $A_j\stackrel{(\frak{a})}{\mapsto} \frak{a}_j(k):=\frak{a}(\nu  (k-\xi_j))$. \medskip

Next, we study the following functions \[\frak{O}_j(k):=\prod_{0\le \ell\le j}\frak{a}_\ell(k)\,,\quad j\in \{0,\ldots,\nu  -1\}\,,
\] which are outer in $k\in \C^+$. Specifically, let us introduce
\[
\mathfrak{F}_\nu  (\xi):=\sup_{k\in S_\xi,0\le j\le \nu  -1}|\arg \,\frak{O}_j(k)|=\sup_{k\in S_\xi,0\le j\le \nu  -1}\left|\sum_{\ell=0}^j\arg \,\frak{a}_\ell(k)\right|\,.
\]
\begin{lemma}\label{llog} We have 
\begin{equation}\label{loga}
\mathfrak{F}_\nu  (\xi)\gtrsim \log\nu  
\end{equation}
for $\xi\in [\frac 12,1]$.\label{t1}
\end{lemma}
\begin{proof}
Notice that 
\[
\arg \frak{O}_j(k)=\Im \left(\log \frak{a}_0(k)+\ldots+\log \frak{a}_{j}(k)\right)=\sum_{p=0}^j\Im \left(\frac{1}{\pi i}\int\frac{\log|\frak{a}(\nu  (s-\xi_p))|}{s-k}ds\right), \quad k\in \C^+\,.
\]

%\qbezier(30,30)(31,32)(35,80)(50,32)(80,30)

\begin{tikzpicture}
    % Draw a line from (0,0) to (2,2)
    \draw[thick, black] (-0.1,0) -- (13,0);
    %\draw[thick, black] (5.5,-0.1) -- (5.5,4);
    % Draw a circle at (3,1) with radius 0.5cm
    \draw[black, fill=black] (1,0) circle (0.4mm);
    \node at (1,-0.4) {$\xi_{j-4}$};
    \draw[black, fill=black] (2,0) circle (0.4mm);
    \node at (2,-0.4) {$\xi_{j-3}$};
    \draw[black, fill=black] (3,0) circle (0.4mm);
    \node at (3,-0.4) {$\xi_{j-2}$};
    \draw[black, fill=black] (4,0) circle (0.4mm);
    \node at (4,-0.4) {$\xi_{j-1}$};
    \draw[black, fill=black] (5,0) circle (0.4mm);
    \node at (5,-0.4) {$\xi_j$};
    \node at (6,0.4) {$k$};
    \draw[black, fill=black] (5.8,0.2) circle (0.4mm);
    \node at (8,-1) {Figure 3. Creation of logarithmic growth by piling bumps to the left of $k$.};
     \node at (8,-1.5) { The ``height'' of each bump is $\sim 1$ and its ``width'' is $\sim 1/\nu  $. We have $\xi_j-\xi_{j-1}=1/\nu  $\,.};

    \foreach \r in {0,...,4}
    {
    \draw[darkbrown, samples=1000, domain=-1:1, smooth] plot [smooth, solid] coordinates {(\r+0.6,0.1) (\r+0.8, 0.6) 
    (\r+1,4) (\r+1.2,0.6) (\r+1.4,0.1)};
    }

    \end{tikzpicture}
\bigskip

Notice that $\frak{a}(\xi)$ is smooth and 
\begin{equation}\label{stro}
0\le \log |\frak{a}(\xi)|\le_{\ell} |\xi|^{-\ell},\quad |\xi|\to\infty
\end{equation}
 for every $\ell\in \mathbb{N}$. Then, for $j\in \{\frac \nu  2,\ldots, \nu  -1\}$, we get
\[
\int_\R\left(\sum_{p=0}^{j}\log |\frak{a}(\nu  (s-\xi_p))|\right)\frac{1}{s-k}ds=\int_{\R}\left(\sum_{p=0}^j \log |\frak{a}(s'-p)|\right) \frac{1}{s'-\nu   k}ds'\,.
\]
We claim that
\begin{equation}\label{hyu}
\left| \int_{\R}\left(\sum_{p=0}^j \log |\frak{a}(s'-p)|\right) \Re \left(\frac{1}{s'-\nu   k}\right)ds'\right|\ge C_1|\log |k-\xi_j||-C_2
\end{equation}
for all $k$ that satisfy  $0\le \Im k\le 1$ and $\xi_j+\frac{1}{2\nu  }\le \Re k\le 4$ (see the Figure 3). Indeed, if $\omega:=\nu   k=:x+iy$, then $x\ge j+\tfrac 12$ and
\[
\Re \left(\frac{1}{s'-\omega}\right)=\frac{s'-x}{(s'-x)^2+y^2}\,.
\]
The nonnegative function $\sum_{p=0}^j \log |\frak{a}(s'-p)|$ satisfies
\[
\sum_{p=0}^j \log |\frak{a}(s'-p)|\stackrel{\eqref{stro}}{\le}_\ell (1+|s'-j|)^{-\ell},\quad \forall \ell\in \mathbb{N}
\]
for $s'>j$. Therefore, 
\[
\left| \int_{s'>j, |s'-x|>1}\left(\sum_{p=0}^j \log |\frak{a}(s'-p)|\right) \frac{s'-x}{(s'-x)^2+y^2}ds'\right|<C
\]
if $x\ge j+\tfrac 12$. Moreover,  $|\partial_{s'}\sum_{p=0}^j \log |\frak{a}(s'-p)||\le C$ for $s'>j$, so we have 
\[
\left| \int_{s'>j, |s'-x|<1}\left(\sum_{p=0}^j \log |\frak{a}(s'-p)|\right) \frac{s'-x}{(s'-x)^2+y^2}ds'\right|<C
\]
for $x\ge j+\tfrac 12$. Next, we again assume that $x\ge j+\tfrac 12$ and consider the third integral (check Figure 3):
\begin{eqnarray*}
\left| \int_{s'<j}\left(\sum_{p=0}^j \log |\frak{a}(s'-p)|\right) \frac{x-s'}{(s'-x)^2+y^2}ds'\right|\gtrsim \int_0^j\frac{x-s'}{(x-s')^2+y^2}ds'=\\
\left|\Re \int_0^j\frac{1}{\omega-s'}ds'\right|=\left|\Re \log \left(\frac{\omega-j}{\omega}\right)\right|=|\log|(k-\xi_j)/k||\ge |\log|k-\xi_j||-C
\end{eqnarray*}
for the given range of $k$ (we have taken the principal branch of the  logarithm). Combining these bounds, we have \eqref{hyu}.\medskip

Now that we proved the claim \eqref{hyu}, we get \[\max_{j\le \nu  -1} |\arg\, \frak{O}_j(k)|\ge C_1|\log(\tfrac 1\nu  +\Im k)|-C_2\] for all $k$ that satisfy $0\le \Im k\le 1$ and $\frac 12\le \Re k\le 1$. That implies \eqref{loga}.
\end{proof}

\noindent {\bf Second step: well-separated bumps.}  Besides $\nu  $, we now take another large parameter $R\ge \nu  $ and let
\[
Q_{\nu  ,R}(x)=\sum_{j=0}^{\nu  -1}A_j(x-j(\nu  +R))\,,
\]
where $A_j$ was introduced in \eqref{lo8}.

In other words, we place $\nu  $ bumps $A_j$, each of which is supported inside $(0,\nu  )$ on $\R^+$, separating them by the intervals of size $R$ where  $ Q_{\nu  ,R}$ vanishes. By the group property for the fundamental matrix, we have 
\begin{equation}\label{jop}
\left(
\begin{array}{cc}
\frak{A}^*((\nu  +R)\ell,k,Q_{\nu  ,R})& \frak{B}^*((\nu  +R)\ell,k,Q_{\nu  ,R})\\
\frak{B}((\nu  +R)\ell,k,Q_{\nu  ,R}) & \frak{A}((\nu  +R)\ell,k,Q_{\nu  ,R})
\end{array}
\right)=\prod_{j=0}^{\ell-1}\left(
\begin{array}{cc}
e^{ikR}\frak{a}^*(k,A_j)& e^{ikR}{\frak{b}^*(k,A_j)}\\
\frak{b}(k,A_j) & \frak{a}(k,A_j)
\end{array}
\right)
\end{equation}
for every $\, \ell\in \{1,\ldots,\nu  \}, \,k\in \C^+$, where $\frak{a}^*(k,A_j)=\frak{A}^*(\nu  ,k,A_j)=e^{i\nu   k}\overline{\frak{a}(\overline{k},A_j)}$ and $\frak{b}^*(k,A_j)=\frak{B}^*(\nu  ,k,A_j)=e^{i\nu   k}\overline{\frak{b}(\overline{k},A_j)}$.\smallskip

We will need to use the following bounds later. Since $\|A_j\|_1\sim 1$,  Lemma \ref{ell1} from the Appendix yields
\[
\sup_{x\in \R^+, \Im k\ge 0}\|X(x,k,A_j)\|\lesssim 1\,.
\]
Therefore, 
$
|\frak{a}(k,A_j)|\lesssim 1, \,|\frak{a}^*(k,A_j)|\lesssim 1,\,|\frak{b}(k,A_j)|\lesssim 1,\,|\frak{b}^*(k,A_j)|\lesssim 1
$
uniformly in $k\in \C^+$, in $j\in \{0,\ldots,\nu  -1\}$, and in $\nu  \in [1,\infty)$. From \eqref{lo9}, also proved in the Appendix, we get
\[
|\frak{A}((\nu  +R)\ell,k,Q_{\nu  ,R})-\prod_{j=0}^{\ell-1}\frak{a}(k,A_j)|\le e^{-R\Im k}C_1^\ell
\]
uniformly in $k: \Im k\ge 0$ and $\ell\in \{1,\ldots,\nu  \}$.  For example, take $R=\nu  ^3$ to guarantee
\[
|\frak{A}((\nu  +R)\ell,k,Q_{\nu  ,\nu  ^3})-\prod_{j=0}^{\ell-1}\frak{a}(k,A_j)|<e^{-C_2\nu  }, \quad \ell\in \{1,\ldots,\nu  \}
\]
in the domain $\Im k\ge \nu  ^{-1}$ with some positive $C_2$. 

That inequality implies that if we take any ray $\{k=\xi+i\eta, \eta\ge \nu  ^{-1}\}$ and consider two curves $\{\frak{A}((\nu  +R)\ell,k,Q_{\nu  ,\nu  ^3})\}$ and $\{\prod_{j=0}^{\ell-1}\frak{a}(k,A_j)\}$ which are the images of such ray under these two maps, then the points on these curves, corresponding to the same $k$, will be at most $e^{-C_2\nu  }$ apart. Notice that  both curves approach the point $z=1$ when $\eta\to+\infty$ and they stay outside the unit disc around the origin. Comparing with the results  for the simplified decoupled model, we get
\[
\frak{M}(\xi,Q_{\nu  ,\nu  ^3})\stackrel{\eqref{loga}}{\gtrsim} \log\nu  
\]
for $\xi\in [\frac 12,1]$ as long as $\nu  $ is large enough. Now, notice that $\|Q_{\nu  ,\nu  ^3}\|_2\sim 1$. Finally, taking $\Psi_n:=Q_{n,n^3}$ one gets \eqref{gg1}. \end{proof}
\noindent {\bf Remark. } Notice that in the statement of the theorem, we can make $\sup_n\|\Psi_n\|_2$ as small as we like if we start with $A_\delta=\delta A$ instead of $A$ and make a positive $\delta$ small.\medskip

 In the next section, it will be convenient to first address \hyperlink{id1}{Q1} and the first part of \hyperlink{id2}{Q2} in the context of the Dirac equation on the whole real line. After we answer these questions for Dirac model, we will come back to the Krein systems.\bigskip

\section{Dirac operator formalism, inverse scattering theory approach, and negative answers to {Q1} and Q2}

We recall some basics about the Dirac equation on the line closely following \cite{BD}, Section 2 (see also \cite{FD}). Given a complex-valued function $q \in \Sch(\R)$, define the  differential operator
\begin{equation}\label{dir1}
\mathcal{L}_q = i\sigma_3\partial_x + i(q\sigma_- - \ov{q} \sigma_+),
\end{equation}
where $\sigma_3$, $\sigma_{\pm}$  
\begin{equation}\label{cmatrices}
\sigma_3 =  \begin{pmatrix}1 & 0 \\ 0 & -1\end{pmatrix}, 
\qquad
\sigma_+ = \begin{pmatrix}0&1\\0&0\end{pmatrix},
\qquad
\sigma_- = \begin{pmatrix}0&0\\1&0\end{pmatrix}.
\end{equation}
Let us also define 
$$
E(x, k) = e^{\frac{k}{2i}x\sigma_3} = \begin{pmatrix}e^{\frac{k}{2i}x} & 0 \\ 0 & e^{-\frac{k}{2i}x}\end{pmatrix}. 
$$
In the free case when $q=0$, the matrix-function $E$ solves $\mathcal{L}_0E={\frac{k}{2}} E,\, E(0,k)=\idm$. 
Since $q \in \Sch(\R)$, it decays at infinity fast and therefore one can find two solutions
 $T_{\pm} = T_{\pm}(x,\xi,q)$ such that 
\begin{equation}\label{eq1lol}
\mathcal{L}_q T_{\pm} = \tfrac{\xi}{2}T_{\pm}, \qquad T_{\pm} = E(x,\xi) + o(1), \qquad x \to \pm\infty,
\end{equation}
for every $\xi\in \R$.
These solutions are called the {\it Jost solutions} for $\mathcal{L}_q$. Since both $T_{+}$ and $T_{-}$ solve the same ODE, they must satisfy
\begin{equation}\label{fix1}
T_{-}(x,\xi,q) =  T_{+}(x,\xi,q)T(\xi,q), \qquad x \in \R, \qquad \xi\in \R,
\end{equation}
where the matrix $T = T(\xi,q)$ does not depend on $x \in \R$.  One can show that it has the form
\begin{equation}\label{eq4}
T(\xi,q) = 
\begin{pmatrix}
a(\xi,q) & \ov{b(\xi,q)} \\
b(\xi,q) & \ov{a(\xi,q)}
\end{pmatrix}, \qquad \det T = |a|^2 - |b|^2 = 1.
\end{equation}
The matrix $T$ is called the {\it reduced transition matrix} for $\mathcal{L}_q$, and the ratio $\rc = b/a$ is called the {\it reflection coefficient} for $\mathcal{L}_q$ and $1/a$ is called the {\it transmission coefficient}. One can obtain $T$ in a different way: let $Z= Z(x, k,q)$, $x \in \R$, $k \in \C$ be the fundamental matrix  for $\mathcal{L}_q$ with spectral parameter $\tfrac k2$, that is, 
\begin{equation}\label{eq2}
\mathcal{L}_q Z = \tfrac{k}{2}Z, \qquad Z(0, k,q) = \idm. 
\end{equation}
Then, we have
%$T^{-1}_\pm(0,k)=T^{-1}_\pm(\xi,k)Z_q(\xi,k)$ 
\begin{equation}\label{lk1}
Z(x,k,q) = T_\pm(x,k,q)T^{-1}_\pm(0,k,q)
\end{equation}
and the pointwise limits
\begin{equation}\label{eq3}
T^{-1}_\pm(0,\xi,q) = \lim_{x \to \pm\infty}E^{-1}(x,\xi)Z(x,\xi,q)
\end{equation}
exist for every $\xi \in \R$. Moreover, we have $T(\xi,q) = T_+^{-1}(0,\xi,q)T_-(0,\xi,q)$ on $\R$. 

\medskip

The coefficients $a,b$, and $\rc$ were defined for $\xi\in \R$ and they satisfy $|a|^2=1+|b|^2$, $1-|\rc|^2=|a|^{-2}$ for these $\xi$. However, one can show that $a(\xi,q)$ is the boundary value of the outer function defined in $\C^+$ by the formula 
\[
a(k,q)=\exp \left(\frac{1}{\pi i}\int_{\R}\frac{1}{\xi-k} \log|a(\xi,q)|\,d\xi\right), \quad k\in \C^+,
\]
which, in view of identity $1-|\rc|^2=|a|^{-2}=(1+|b|^2)^{-1}$ on $\R$, can be written as
\begin{equation}\label{ar}
a(k,q)=\exp \left(-\frac{1}{2\pi i}\int_{\R}\frac{1}{\xi-k} \log(1-|\rc(\xi)|^2)d\xi\right)=\exp \left(\frac{1}{2\pi i}\int_{\R}\frac{1}{\xi-k} \log(1+|b(\xi,q)|^2)d\xi\right)\,.
\end{equation}
That shows, in particular, that $b$ defines both $a$ and $\rc$, and $\rc$ defines $a$ and $b$. The {\it sum rule} for the Dirac equation reads
\[
\pi \int_\R |q(x)|^2dx=\int_\R \log |a(\xi,q)|d\xi\,.
\]
The map $q(x)\stackrel{[r]}{\mapsto}\rc(k,q)$ is called the direct scattering transform and its inverse is called the inverse scattering transform. These maps are well-studied when $q\in \Sch(\R)$. In particular, we have the following result (see  Theorem 2.1 in \cite{BD}, the proof can be found in \cite{FD}):
\begin{Thm}\label{Thm1}
The map $q \stackrel{[r]}{\mapsto} \rc$ is a bijection from $\Sch(\R)$ onto the set of complex-valued functions $\{\rc\in\Sch(\R), \|\rc\|_{L^\infty(\R)}<1\}$. 
\end{Thm}
In particular, each $b(k)\in \Sch(\R)$ uniquely defines $q\in \Sch(\R)$ such that $b(k)=b(k,q)$.
That scattering transform has some symmetries (see \cite{BD} and \cite{kov1}, p.241): 

\begin{Lem}\label{sym} If $q\in \Sch(\R)$ and $\xi \in \R$, then
\begin{eqnarray*}
 \text{\it (dilation):} & \qquad \rc(\xi,\mu q(\mu x))=\rc(\mu^{-1}\xi,q(x)), \quad \mu>0\,,
\\
\text{\it (conjugation):}& \qquad \rc(\xi,\overline{q}(x))={\ov{\rc(-\xi,q)}}\,,
\\
\text{\it (translation):}& \qquad \rc(\xi,q(x-\ell))=\rc(\xi,q(x))e^{-i\xi \ell}, \quad \ell\in \R\,,
\\
\text{\it (modulation):}&  \qquad \rc(\xi,e^{-i\beta x}q(x))=\rc(\xi+\beta,q(x)), \quad \beta\in \R\,,
\\
\text{\it (rotation):}&  \qquad \rc(\xi,\zeta q(x))={\zeta}\rc(\xi,q(x)), \quad \zeta\in \C, \; |\zeta|=1\,.
\end{eqnarray*}
\end{Lem}
They imply the corresponding formulas for $a(\xi,q)$ and $b(\xi,q)$, too.
One advantage of working with Dirac operators on the whole line is the presence of translation symmetry, which we do not have for the Krein system on $\R^+$.
The other benefit of the Dirac model is Theorem \ref{Thm1} which does not require the scattering data $\rc$ to be analytic in $\C^+$.

We need to outline the connection between the Krein systems and the Dirac operator $\mathcal{L}_q$ (see \cite{BD},~p.225). Given $q$, let 
\[
A_+(x)=-\tfrac 12\overline{q(\tfrac x2)}, \quad A_-(x)=\tfrac 12{q(-\tfrac x2)}\,,\quad  x\ge 0\,.
\]
For these two Krein systems, we consider $\frak{a}(k,A_+),\frak{b}(k,A_+)$ and $\frak{a}(k,A_-),\frak{b}(k,A_-)$\,. Then, we have (check \cite{BD}, p.227 and \cite{kov1}, p.241)
\begin{equation}\label{trans1}
\left\{\begin{array}{c}a({2}k,q)=\frak{a}(k,A_+) \frak{a}(k,A_-) - \frak{b}(k,A_+)\frak{b}(k,A_-),\,\\ b({2}\xi,q)= \frak{a}(\xi,A_-)\ov{\frak{b}(\xi,A_+)} - \frak{b}(\xi,A_-)\ov{\frak{a}(\xi,A^+)}\,,
\end{array}\right.
\end{equation}
where $
 k\in \C^+,\,\xi\in \R$.\medskip

The existence of Jost functions can be proved for $q\in L^1(\R)$, too. We will need two results.
\begin{lemma} Consider $v_1,v_2\in L^1(\R)$, both supported inside $[-r,r], r>0$. Take $R\ge 2r$ and let $Q=v_1(x)+v_2(x-R)$. Then, for $\xi\in \R$, we get
\begin{equation}\label{jh0}
\left\{
\begin{array}{c}
a(\xi,Q)=a(\xi,v_1)a(\xi,v_2)+b(\xi,v_1)\overline{b(\xi,v_2)}e^{-\xi R/i},\\ b(\xi,Q)=\overline{a(\xi,v_2)}b(\xi,v_1)+{a(\xi,v_1)}b(\xi,v_2)e^{\xi R/i}\,.
\end{array}
\right.
\end{equation}
\end{lemma}
\begin{proof}
This is immediate from the definition and the group property of the fundamental matrix. Alternatively, one can see it from \eqref{jop}, where $\ell=2$, and \eqref{trans1}.
\end{proof}
\begin{lemma} \label{kot1} Consider $v_1,v_2\in L^1(\R)$. Take $R\ge 0$ and let $Q=v_1(x)+v_2(x-R)$. Then, 
\begin{eqnarray}\label{jh1}
\left\{
\begin{array}{c}
a(\xi,Q)=a(\xi,v_1)a(\xi,v_2)+b(\xi,v_1)\overline{b(\xi,v_2)}e^{-\xi R/i}+O(\eps^*(R))\\ b(\xi,Q)=\overline{a(\xi,v_2)}b(\xi,v_1)+{a(\xi,v_1)}b(\xi,v_2)e^{\xi R/i}+O(\eps^*(R))
\end{array}
\right.,\\
\eps^*(R):=\left(\int_{|x|>R/2}|v_1(s)|ds+\int_{|x|>R/2}|v_2(s)|ds\right)\exp(C(\|v_1\|_1+\|v_2\|_1))\nonumber
\end{eqnarray}
uniformly  in $\xi\in \R$.
\end{lemma}
\begin{proof}
We can take  $v_1^{(R)}:=v_1\cdot\chi_{|x|<R/2}, v_2^{(R)}:=v_2\cdot\chi_{|x|<R/2}$ and use \eqref{stab1} to write
\begin{equation}\label{qua1}
|a(\xi,v_j^{(R)})-a(\xi,v_j)|\lesssim \eps_j(R), \quad |b(\xi,v_j^{(R)})- b(\xi,v_j)|\lesssim \eps_j(R), \quad j\in \{1,2\}
\end{equation}
uniformly in $\xi\in \R$, where
\[
\eps_j(R):=\left(\int_{|x|>R/2}|v_j(s)|ds\right)\exp(C\|v_j\|_1)\,.
\]
Similarly, if one defines $Q^{(R)}:=v_1^{(R)}(x)+v_2^{(R)}(x-R)$, then \eqref{stab1} provides
\begin{equation}\label{qua2}
|a(\xi,Q^{(R)})-a(\xi,Q)|\lesssim \eps^*(R), \quad |b(\xi,Q^{(R)})- b(\xi,Q)|\lesssim \eps^*(R)\,.
\end{equation}
We can apply the previous lemma to $Q^{(R)}$. Then, substituting the obtained bounds along with (check \eqref{sto} and \eqref{cod7}) the estimates
\[
|a_j(\xi,v_j)|\lesssim \exp(C\|v_j\|_1), \quad |b_j(\xi,v_j)|\lesssim \exp(C\|v_j\|_1)\,,
\]
into \eqref{jh0}, we arrive at \eqref{jh1}.
\end{proof}

The strong maximal function can be defined for the Dirac operator $\mathcal{L}_q$ as follows. Let $T_1,T_2\in \R, T_1<T_2$.  Then,
\[
\mathrm{M}(\xi,q):=\sup_{k\in S_\xi, T_1,T_2}|\arg\, a(k,q\cdot \chi_{T_1\le x\le T_2})|\,,
\]
where $S_\xi$ is a Stolz angle at $\xi\in \R$. We can let $q\in L^2(\R)$ and ask the same questions we asked for Krein systems, i.e., whether the maximal function $\mathrm{M}$ is well-defined and whether $a(k,q\cdot \chi_{T_1\le x\le T_2})\to a(k,q)$ pointwise when $T_1\to -\infty, T_2\to +\infty$. In fact, taking $q=0$ for $x\le 0$ in the formula \eqref{trans1}, we already see that the weak-type bound for $\mathrm{M}(\xi,q)$ fails due to Theorem \ref{ks12}. 
If we consider another maximal function
\[
\mathrm{M}_w(\xi,q)=\sup_{T_1,T_2}|\arg\, a(\xi,q\cdot \chi_{T_1\le x\le T_2})|\,, \quad \xi\in \R\,,
\]
then it is clear that $\mathrm{M}_w(\xi,q)\le \mathrm{M}(\xi,q)$.
The next result shows that these maximal functions $\mathrm{M}_w(\xi,q)$ and $\mathrm{M}(\xi,q)$ might  not be even defined for $q\in L^2(\R)$.

\begin{theorem} There is $q\in L^2(\R)$ such that $\mathrm{M}_w(\xi,q)=+\infty$ for $\xi\in [\frac 12,1]$.
\end{theorem}
\begin{proof}
We do the proof in two steps. \medskip

\noindent {\bf First step: building a sample.} Unlike in the first section, we use inverse scattering theory as a stepping stone in our construction. Consider a function $b(\xi)\in C_c^\infty(\R)$ with support inside the interval $(0,1)$ which is not identically equal to zero. Take $\delta\in (0,1)$ and consider $b_{[\delta]}(\xi):=\delta b(\xi)$. It will define $a_{[\delta]}(k)$ by \eqref{ar}, $\rc_{[\delta]}(\xi)\in \Sch(\R)$ and $q_{[\delta]}(x)\in \Sch(\R)$. For $q_{[\delta]}$, we have a bound
\begin{equation}\label{lerch4}
|\partial^pq_{[\delta]}(x)|\le_{\ell,p} (x^2+1)^{-\ell/2}, \quad \forall \ell\in \mathbb{N}, \, p\in \Z^+\,.
\end{equation}
By the sum rule, 
\begin{equation}\label{dn1}
2\pi \int_\R |q_{[\delta]}|^2dx= \int_\R \log (1+|b_{[\delta]}|^2)d\xi\sim \delta^2\,.
\end{equation}
Take $\nu  \ge 1$ and let $q_{[\delta]}^{(1/\nu  )}(x):=q_{[\delta]}(x/\nu  )/\nu  $. Then, 
$
b(\xi,q_{[\delta]}^{(1/\nu  )})=b(\nu  \xi,q_{[\delta]})=\delta b(\nu  \xi)\,.
$
For every $j\in \{0,\ldots,\nu  -1\}$, we let $q_j(x):=q_{[\delta]}^{(1/\nu  )}(x)e^{i\xi_j x}$. Notice that 
\begin{equation}\label{fega}
|q_j(x)|\le_{\ell} \nu  ^{-1}((x/\nu  )^2+1)^{-\ell/2}, \quad \forall \ell\in \mathbb{N}\,.
\end{equation}
The modulation property of the scattering transform yields $b(\xi,q_j)=\delta b(\nu  (\xi-\xi_j)), a(\xi,q_j)=a_{[\delta]}(\nu  (\xi-\xi_j)), \xi_j:=j/\nu  $.  Take $\rho>0$, let
\[
Q_{\nu  ,m,\rho,\delta}(x):=\sum_{j=0}^{m} q_j(x-\rho j), \, m\in \{0,\ldots,\nu  -1\}
\]
and choose $\rho=e^{\nu  ^2}$. 
By construction, we have 
\begin{eqnarray}
b(\xi,q_j)b(\xi,q_\ell)=0, \quad j\neq \ell\,,   \label{pool1}\\
\sum_{j=0}^m |b(\xi,q_j)|\lesssim \delta, \quad\sum_{j=0}^m |b(\xi,q_j)|^2\lesssim \delta^2,\,\quad \forall \xi\in \R\,.\label{pool2}
\end{eqnarray}
We can now apply the previous Lemma $\nu  $ times recursively taking each time $v_1=Q_{\nu  ,m,\rho,\delta}$, $v_2=q_{m+1}(x)$, and $R=\rho(m+1)$. To estimate the errors, we notice that $\|Q_{\nu  ,m,\rho,\delta}\|_1\lesssim m\le \nu  $ and $\|q_{m+1}\|_1\lesssim  1$. The strong decay \eqref{fega} of Schwartz-class functions $v_1$ and $v_2$ gives
$
\eps^*\le  e^{-2\nu  ^2}
$
for each iteration and $\nu  \ge \nu  _0$ with large enough $\nu  _0$. Rewrite \eqref{jh1} as
\begin{eqnarray*}
\left(
\begin{array}{cc}
a(\xi,Q_{\nu  ,m+1,\rho,\delta}) & \bar{b}(\xi,Q_{\nu  ,m+1,\rho,\delta})  \\
b(\xi,Q_{\nu  ,m+1,\rho,\delta}) & \bar{a}(Q_{\nu  ,m+1,\rho,\delta})
\end{array}
\right)=\hspace{5cm}\\
\left(
\begin{array}{cc}
a(\xi,q_{m+1}) & \bar{b}(\xi,q_{m+1})e^{-\xi \rho (m+1)/i}\\
b(\xi,q_{m+1})e^{\xi \rho (m+1)/i} & \bar{a}(\xi,q_{m+1})
\end{array}
\right)
\left(
\begin{array}{cc}
a(\xi,Q_{\nu  ,m,\rho,\delta}) & \bar{b}(\xi,Q_{\nu  ,m,\rho,\delta})\\
b(\xi,Q_{\nu  ,m,\rho,\delta}) & \bar{a}(Q_{\nu  ,m,\rho,\delta})
\end{array}
\right)+O(\eps^*)\,.
\end{eqnarray*}
Recall that $\rho=e^{\nu^2}$. Applying Lemma \ref{lap5}, we get
\begin{eqnarray*}
\left(
\begin{array}{cc}
a(\xi,Q_{\nu  ,m,e^{\nu^2},\delta}) & \bar{b}(\xi,Q_{\nu  ,m,e^{\nu^2},\delta})\\
b(\xi,Q_{\nu  ,m,e^{\nu^2},\delta}) & \bar{a}(Q_{\nu  ,m,e^{\nu^2},\delta})
\end{array}
\right)=\hspace{4cm}\\
\prod_{n=0}^{m}\left(
\begin{array}{cc}
a(\xi,q_{n}) & \bar{b}(\xi,q_{n})e^{-\xi n e^{\nu^2}/i}\\
b(\xi,q_{n})e^{\xi n e^{\nu^2}/i} & \bar{a}(\xi,q_{n})
\end{array}
\right)+O(e^{-\nu  ^2})\,.
\end{eqnarray*}
We multiply the matrices in the product using \eqref{pool1} to get
\begin{equation}\label{choi}
a(\xi,Q_{\nu  ,m,e^{\nu^2},\delta})=a(\xi,q_0)\cdot\ldots \cdot a(\xi,q_{m})+O(e^{-\nu  ^2})
\end{equation}
uniformly in $\xi\in \R, \delta\in (0,1)$, and  $\nu  \ge \nu  _0$.
Moreover, 
\begin{eqnarray*}
1\le |a(\xi,q_0)\cdot\ldots \cdot a(\xi,q_{m})|=\exp\left(\tfrac 12\sum_{j=0}^m \log(1+|b(\xi,q_j)|^2) \right)\le\\
\exp\left(\tfrac 12\sum_{j=0}^m |b(\xi,q_j)|^2 \right)\stackrel{\eqref{pool2}}{\le}
\exp(C\delta^2)\le 1+C\delta^2\,.
\end{eqnarray*}
Hence, we get
\begin{eqnarray}\label{glob1}
1\le |a(\xi,Q_{\nu  ,m,e^{\nu^2},\delta})|\le  1+C\delta^2+e^{-\nu  ^2}
\end{eqnarray}
and an identity
 \[
|b(\xi,Q_{\nu  ,m,e^{\nu  ^2},\delta})|^2=|a(\xi,Q_{\nu  ,m,e^{\nu  ^2},\delta})|^2-1
\]
implies
\begin{eqnarray}\label{lor8}
|b(\xi,Q_{\nu  ,m,e^{\nu  ^2},\delta})|\lesssim \delta+e^{-\nu  ^2/2}
\end{eqnarray}
for all $\xi\in \R,\delta\in (0,1)$, and $m\in \{0,\ldots,\nu  -1\}$. Recall that $|a(k,V)-1|\to 0$ when $\Im k\ge 0, |k|\to\infty$ and $V\in \Sch(\R)$. Therefore, the application of the maximum principle to the function \[a(k,Q_{\nu  ,m,e^{\nu  ^2},\delta})-a(k,q_0)\cdot\ldots \cdot a(k,q_{m})\,,\] analytic in $k\in \C^+$, gives 
\[
|a(k,Q_{\nu  ,m,e^{\nu  ^2},\delta})-a(k,q_0)\cdot\ldots \cdot a(k,q_{m})|\stackrel{\eqref{choi}}{\le} Ce^{-\nu  ^2}\,,
\]
when $k\in \C^+$.
By the reasoning used in the proof of Lemma \ref{llog}, we get (now we have that the height of each bump is $\sim \delta^2$ and its width, as before, is $\sim \nu  ^{-1}$):
\begin{equation}\label{gr4}
\max_{\nu  /2\le m\le \nu  -1,\xi\in [\frac 12,1]}|\arg (a(\xi,q_0)\cdot \ldots \cdot a(\xi,q_{m}))|\ge  C_1\delta^2\log\nu  
\end{equation}
and, therefore, 
\begin{equation}\label{lotg}
\max_{\nu  /2\le m\le \nu  -1,\xi\in [\frac 12,1]}
|\arg a(\xi,Q_{\nu  ,m,e^{\nu  ^2},\delta})|\gtrsim \delta^2\log\nu  \,,
\end{equation}
provided that both $\delta^2\log\nu  $ and $\nu  $ are large enough and the bound
\begin{equation}\label{condi9}
C_1\delta^2\log\nu  -Ce^{-\nu  ^2}>C_2\delta^2\log\nu  
\end{equation} holds with some positive constant $C_2$.
We will call $Q_{\nu  ,\nu  -1,e^{\nu  ^2},\delta}$ {\it a sample}. This is not a compactly supported function but it is essentially localized to the interval, e.g., $[-\nu  ^2,\nu   e^{\nu  ^2}]$ in a sense that its size is negligible outside that range when $\nu  \to\infty$. For the $L^2$-norm of the sample, we recall \eqref{dn1} and sparseness of bumps within a sample to get
\begin{equation}\label{bvi}
\|Q_{\nu  ,\nu  -1,e^{\nu  ^2},\delta}\|_2^2\sim \delta^2\,.
\end{equation}

\noindent {\bf Remark.} By taking $\nu\to +\infty$, our argument already shows that the weak-type bound
\begin{equation}\label{we7}
|\{\xi: \mathrm{M}_w(\xi,q)\ge \lambda \}|\lesssim \frac{\|q\|_2}{\lambda}, \, \forall \lambda>0
\end{equation}
fails even when $q$ is restricted to $\|q\|_2\le \delta$ and $\delta$ is an arbitrarily small number.

\smallskip
\noindent {\bf Second step: putting samples together.} In the previous construction of samples, we choose 
\begin{equation}\label{choice}
\nu  _n=\exp(\exp(n^2)),\quad \delta_n=\exp(-n),\quad n\in \mathbb{N}
\end{equation}
where $n\ge n_0$. This choice of parameters will satisfy \eqref{condi9} for $n\ge n_0$, where $n_0$ is sufficiently large, because $\delta_n^2\log\nu_n=e^{-2n+n^2}\to +\infty$ as $n\to\infty$. We also take \[t_n=\exp(\exp(\nu  _n))\] and construct the potential $q$ by spreading the samples over the line at distance $t_n$ from the origin:
\[
q(x):=\sum_{m\ge n_0} Q_m(x), \quad Q_m(x):=Q_{\nu  _m,\nu  _m-1,e^{\nu  _m^2},\delta_m}(x-t_m),\quad q^{(n)}(x):=\sum_{n_0\le m\le n} Q_m(x)\,.
\]
Given that $t_{n+1}/t_n\to\infty$ and $t_n/(\nu  _ne^{\nu  _n^2})\to\infty$ as $n\to\infty$, our $q$ is a very sparse configuration of samples. Moreover, the $n$-th sample is ``damped'' by a small parameter $\delta_n$ to control the $L^2$-norm of the resulting $q$. Indeed, \eqref{bvi} yields
 \[\|q\|_2^2 \lesssim \sum_{n\ge n_0} \exp(-2n)\lesssim e^{-2n_0}.\] 
Notice that  we can make $\|q\|_2$ as small as we want by choosing $n_0$ large enough.\smallskip
 
 Next, we use \eqref{jh1} again to analyze the fundamental matrix after combining the samples in $q$. 
In fact, after adding a new sample $Q_n$ to those already collected in $q^{(n-1)}$, the  resulting recurrence becomes (see Lemma~\ref{kot1})
\begin{equation}\label{jh2}
\left\{
\begin{array}{c}
a(\xi,q^{(n)})=a(\xi,q^{(n-1)})a(\xi,Q_n)+b(\xi,q^{(n-1)})\overline{b(\xi,Q_n)}e^{-\xi t_n/i}+O(\eps^*_n),\\ b(\xi,q^{(n)})=\overline{a(\xi,Q_n)}b(\xi,q^{(n-1)})+{a(\xi,q^{(n-1)})}b(\xi,Q_n)e^{\xi t_n/i}+O(\eps^*_n)
\end{array}
\right.
\end{equation}
uniformly in $\xi\in \R$. To bound $\eps^*_n$, one uses
\begin{equation}\label{topu1}
\|q^{(n)}\|_1\lesssim n\nu  _{n-1}\,,\quad\|Q_n\|_1\lesssim \nu  _n
\end{equation}
and \eqref{fega} to get
\[
\eps^*_n\lesssim \exp(C(n\nu  _{n-1}+\nu  _n))t_n^{-\tfrac 12}\le t_n^{-\tfrac 13}\,.
\]
According to \eqref{glob1} and \eqref{lor8}, we have bounds \begin{equation}\label{loge1}
 1\le |a(\xi,Q_n)|\le 1+Ce^{-n}, |b(\xi,Q_n)|\le Ce^{-n}.\end{equation} Adding equations in \eqref{jh2} and taking absolute values gives
\[
|a(\xi,q^{(n)})|+|b(\xi,q^{(n)})|\le (1+Ce^{-n})(|a(\xi,q^{(n-1)})|+|b(\xi,q^{(n-1)})|)+O(t_n^{-\tfrac 13}), \, n\ge n_0\,.
\]
Our choice $t_n=\exp(\exp(\nu  _n))$ yields an estimate $|a(\xi,q_n)|+|b(\xi,q_n)|\le 1+Ce^{-n_0}$ after we use the  Lemma~\ref{pron1}. The substitution of that bound back into \eqref{jh2} gives 
\[
|a(\xi,q^{(n)})|\le (1+Ce^{-n})|a(\xi,q^{(n-1)})|+Ce^{-n}, \quad
|b(\xi,q^{(n)})|\le (1+Ce^{-n})|b(\xi,q^{(n-1)})|+Ce^{-n}, \, n\ge n_0\,.
\]
Using Lemma \ref{pron1} one more time  gives
\begin{equation}\label{pron2}
1\le |a(\xi,q^{(n)})|\le 1+Ce^{-n_0}, \quad |b(\xi,q^{(n)})|\le Ce^{-n_0}
\end{equation}
for all $n\ge n_0$. \bigskip\bigskip\bigskip

%So,
%$|a(\xi,q_n)-a(\xi,q_{n-1})a(\xi,Q_n)|\lesssim e^{-n}, \, n\ge n_0+1$. That implies
%\[
%|a(\xi,q_n)-\prod_{s=n_0}^{n}a(\xi,Q_s)|\lesssim e^{-n_0}, \, n\ge n_0\,.
%\]

\begin{tikzpicture}
    % Draw a line from (0,0) to (2,2)
    \draw[thick, black] (-0.1,0) -- (14.5,0);
    %\draw[thick, black] (5.5,-0.1) -- (5.5,4);
    % Draw a circle at (3,1) with radius 0.5cm
    \draw[black, fill=black] (1.4,0) circle (0.4mm);
    \node at (1.4,-0.4) {$t_{n-1}$};
    \draw[darkbrown, samples=1000, domain=-1:1, smooth] plot [smooth, solid] coordinates {(0.6,0.1) (1, 0.2) 
    (1.4,0.5) (1.8,0.2) (2.2,0.1)};
   % \draw[black, fill=black] (5,0) circle (0.4mm);
   % \node at (5,-0.4) {$\xi_j$};
    %\draw[black, fill=black] (5.8,0.2) circle (0.4mm);
    \node at (8,-1) {Figure 4. Putting samples together. Showing the absolute values of bumps};
     \node at (8,-1.5) { in each sample and the place of the cut (vertical segment) at $t_n+(m+0.5)e^{\nu  _n^2}$. };

    \foreach \r in {2,...,3}
    {
    \draw[darkbrown, samples=1000, domain=-1:1, smooth] plot [smooth, solid] coordinates {(4*\r+0.6,0.1) (4*\r+1, 0.2) 
    (4*\r+1.4,0.4) (4*\r+1.8,0.2) (4*\r+2.2,0.1)};
    }
 \draw[black, fill=black] (9.4,0) circle (0.4mm);
    \node at (9.6,-0.4) {$t_{n}$};
    \draw[black, fill=black] (13.4,0) circle (0.4mm);
    \node at (13.6,-0.4) {$t_{n}+m\exp({\nu  _n}^2)$};
   \draw[thick, black] (14,-0.2) -- (14,0.2);
    \end{tikzpicture}
\bigskip

Next we choose each $m\in \{0,\ldots, \nu  _n-1\}$ and consider 
potential $q\cdot \chi_{(-\infty,t_n+(m+0.5)e^{\nu  _n^2}]}$. We apply Lemma \ref{kot1} again choosing $v_1=q_{n-1}$, $v_2=Q_{\nu  _n,m,e^{\nu  _n^2},\delta_n}\cdot \chi_{(-\infty,(m+0.5)e^{\nu  _n^2}]}$, and $R=t_n$.
Then, combining \eqref{topu1},\eqref{loge1}, and \eqref{pron2}, we get
\[
|a(\xi,q\cdot \chi_{(-\infty,t_n+(m+0.5)e^{\nu  _n^2}]})-a(\xi,q_{n-1})a(\xi,v_2)|\lesssim e^{-n_0}\,.
\]
Finally, by applying the same Lemma \ref{kot1} one more time, we obtain
\[
|a(\xi,v_2)-a(\xi,Q_{\nu  _n,m,e^{\nu  _n^2},\delta_n})|\lesssim e^{-\nu  _n^2}
\]
and substitution into the previous estimate, along with \eqref{pron2}, gives us
\[
|a(\xi,q\cdot \chi_{(-\infty,t_n+(m+0.5)e^{\nu  _n^2}]})-a(\xi,q_{n-1})a(\xi,Q_{\nu  _n,m,e^{\nu  _n^2},\delta_n})|\lesssim e^{-n_0}\,.
\]
Application of the maximum principle again yields \[
|a(k,q\cdot \chi_{(-\infty,t_n+(m+0.5)e^{\nu  _n^2}]})-a(k,q_{n-1})a(k,Q_{\nu  _n,m,e^{\nu  _n^2},\delta_n})|\lesssim e^{-n_0}\,,
\]
when $k\in \overline{\C^+}$. Therefore, for $n_0$ large enough, one has
\[
\arg a(k,q\cdot \chi_{(-\infty,t_n+(m+0.5)e^{\nu  _n^2}]})=\arg a(k,q_{n-1})+\arg a(k,Q_{\nu  _n,m,e^{\nu  _n^2},\delta_n})+O(e^{-n_0})
\]
where $k\in \overline{\C^+}$.
Subtracting these equations from each other for two different values of $m$: $m=r$ and $m=0$, we obtain
\begin{eqnarray*}
\max_{\lfloor \nu  _n/2\rfloor\le r\le \nu  _n-1,\xi\in [\frac 12,1]} |\arg a(\xi,q\cdot \chi_{(-\infty,t_n+(r+0.5)e^{\nu  _n^2}]})-\arg a(\xi,q\cdot  \chi_{(-\infty,t_n+0.5e^{\nu  _n^2}]}  )|\\\hspace{4cm}\stackrel{\eqref{lotg}}{\gtrsim} \delta_n^2\log\nu  _n\stackrel{\eqref{choice}}{\gtrsim} e^{-2n+n^2}\,.
\end{eqnarray*}
Taking $n_0$ large enough, that bound implies 
\begin{eqnarray}\label{fafe8}\hspace{1cm}
\max_{\lfloor \nu  _n/2\rfloor\le r\le \nu  _n-1,\xi\in [\frac 12,1]} |\arg a(\xi,q\cdot \chi_{[0,t_n+(r+0.5)e^{\nu  _n^2}]})-\arg a(\xi,q\cdot  \chi_{[0,t_n+0.5e^{\nu  _n^2}]}  )| \gtrsim e^{-2n+n^2}\,.
\end{eqnarray}
Arguing by contradiction, we get $\mathrm{M}_w(\xi,q)=+\infty$ for $\xi\in [\frac 12,1]$.  
\end{proof}

\begin{corollary}\label{cm}
The questions \hyperlink{id1}{Q1} and \hyperlink{id2}{Q2}  for the Dirac equation and the Krein system both have negative answers.\end{corollary}
\begin{proof} It is enough to focus on the Krein systems.
Take $q$ from the previous proof and let
 $A_+(x)=-\tfrac 12\overline{q(\tfrac x2)},\, x\ge 0$. Notice that $\|A_+\|_2<\infty$ and, recalling the formulas \eqref{trans1} and \eqref{fafe8}, one has
\begin{equation}\label{madv}
\sup_{x>0,\xi\in [1/4,1/2]}|\arg \frak{A}(x,\xi,A_+)|=+\infty\,,
\end{equation}
which provides the negative answer to the first part of \hyperlink{id2}{Q2} for Krein systems for both maximal functions.\medskip

The answer to \hyperlink{id1}{Q1} is also negative.   Take $A=A_+$ again.  We argue by contradiction. Indeed, suppose there is $\xi \in [\tfrac 14,\tfrac 12]$ and $\lim_{x\to\infty}\frak{A}(x,\xi,A_+)=:\alpha$. Consider the curve $\Gamma_x$ which is the image of the ray $\{k=\xi+i\eta,\eta\in [0,+\infty)\}$ under the map $k\mapsto \frak{A}(x,k,A_+)$. It connects the point $1$ to the point $\frak{A}(x,\xi,A_+)$ and is outside the disc $\{|z|<1\}$. Clearly, the total variation of the angle (argument) as we follow that curve from $\eta=+\infty$ down to $\eta=0$ is equal to $\arg \frak{A}(x,\xi,A_+)$. That function is continuous in $x$ just like the curve $\Gamma_x$ itself  is continuous in $x$. Take $\epsilon=0.01$ and $x_\epsilon>0$ such that $|\alpha-\frak{A}(x,\xi,A_+)|<\epsilon$ for all $x\ge x_\epsilon$. $\Gamma_x$ is located outside the disc $\{|z|<1\}$ and it is continuous in $x$, which  implies that $|\arg\frak{A}(x,\xi,A_+)-\arg\frak{A}(x_\epsilon,\xi,A_+)|<2\epsilon$ for all $x\ge x_\epsilon$ and that is in contradiction with \eqref{madv}.
\end{proof}

\noindent{\bf Remark.} A simple modification of our construction shows that there is $A\in L^2(\R)$ in the Krein system, for which $\lim_{x\to \infty}\frak{A}(x,\xi,A)$ diverges for all $\xi\in \R$. Indeed, when choosing the $n$-th sample $Q_n$, we can accommodate the growth of $\arg \frak{A}(x,\xi,A)$ for $\xi\in I_n$ where the intervals $\{I_n\}|_{n=1}^\infty$ satisfy $I_n\subset I_{n+1}$ and $\{I_n\}\uparrow \R$. Alternatively, one can chose 
those intervals $\{I_n\}|_{n=1}^\infty$ such that $|I_n|=1$ and the set $\{n: \xi\in I_n\}$ is infinite for every $\xi\in \R$.

 \smallskip

\noindent{\bf Remark.} Let $\epsilon>0$ be an arbitrarily small parameter and   $\nu\in \R^+$ and $ N\in \mathbb{N}$ be two large parameters. Let  $J:=\{j_1<j_2<\ldots<j_N\}$ where $j_n \in \mathbb{Z}, n\in \{1,\ldots,N\}$. Consider the set  $S:=\cup_{n} [j_n/\nu,(j_n+1)/\nu]$ and $S_{\rm int}:=\cup_{n} [(j_n+0.25)/\nu,(j_n+0.75)/\nu]$. The formula \eqref{jh1} and our construction of a sample $Q$ can be used to explicitly build a potential $Q\in \Sch(\R)$ in Dirac equation for which the coefficient $b$ satisfies: 
\[\left\{
\begin{array}{cc}
|b(\xi,Q)|\sim 1, & \xi\in S_{\rm int};\\
|b(\xi,Q)|\le \epsilon, &\xi\in S^c;\\
|b(\xi,Q)|\lesssim 1, & \xi\in \R\,.
\end{array}
\right.
\]
Thus, we can effectively make $Q$ reflectionless on a set $S^c$ of an arbitrary structure while keeping $|r(\xi,Q)|=|b(\xi,Q)/a(\xi,Q)|\sim 1$ on a large subset of $S$.

%\noindent{\bf Remark.}  Consider the Shr\"odinger operator 
%\[
%Hf=-\partial^2_{xx}f+vf
%\]
%on the positive half-line with real-valued potential $v$ and some boundary condition at zero that makes it self-adjoint. The spectrum and generalized eigenfunctions of $H$ were studied for slowly-decaying $v$ in, e.g.,  \cite{ck,kls,Remling}. In particular, for $v\in L^p(\R^+), p\in [1,2)$, the WKB asymptotics of Jost solution was established for a.e. value of $k\in \R$. The modification of our construction for Dirac problem shows that there are $v\in L^2(\R^+)$ for which such Jost asymptotics fails for all $k$ on some interval. 
\bigskip
\section{Appendix}

In this Appendix, we collect some auxiliary statements used in the main text.
\begin{lemma}\label{appro1} Suppose we have square matrices $\{X_1,\ldots,X_\nu  \}$ and $\{\Delta_1,\ldots,\Delta_\nu  \}$. Denote $\Psi_j=X_j\cdot\ldots\cdot X_1$. Assume that $\|X_j\|\le C$ and $\|\Delta_j\|\le\epsilon$ for each $j\in \{1,\ldots,\nu  \}$. Then, we have
\begin{equation}\label{lo9}
\|(X_n+\Delta_n)\cdot\ldots\cdot(X_1+\Delta_1)-\Psi_n\|\le C_1^n\epsilon, \quad n\in \{1,\ldots,\nu  \}\,,
\end{equation}
where $C_1=1+2C+2\epsilon$.
\end{lemma}
\begin{proof}The proof is by induction in $n$. For $n=1$, the bound is correct. Take $n\ge 2$ and suppose we have our inequality for $n-1$, i.e., $\|\Sigma_{n-1}\|\le C_1^{n-1}\epsilon$, where  $\Sigma_{n-1}:=(X_{n-1}+\Delta_{n-1})\cdot\ldots\cdot(X_1+\Delta_1)-\Psi_{n-1}$. Then, 
\[
\Sigma_n=X_n\Sigma_{n-1}+\Delta_n(X_{n-1}+\Delta_{n-1})\cdot\ldots\cdot(X_1+\Delta_1)
\] 
and
$
\|\Sigma_n\|\le CC_1^{n-1}\epsilon+(C+\epsilon)^{n-1}\epsilon\le C_1^n\epsilon
$,
provided that $C_1= 1+2C+2\epsilon$.
\end{proof}\medskip

\begin{lemma}\label{lap5}
 Suppose we have square matrices $\{X_1,\ldots,X_\nu  \}$ and $\{\Delta_1,\ldots,\Delta_{\nu  -1}\}$. Denote $\Psi_j=X_j\cdot\ldots\cdot X_1, \, j\in \{1,\ldots,\nu  \}$. Assume that $\|X_j\|\le C, 1\le C$, and $\|\Delta_j\|\le\epsilon$ for each $j$. Let matrices $\{Y_1,\ldots,Y_\nu  \}$ be defined inductively by:
 \[
 Y_{n+1}=X_{n+1}Y_n+\Delta_n, \quad Y_1=X_1,\quad  n\in \{1,\ldots,\nu  -1\}\,.
 \]
 Then, we have
\begin{equation}\label{lo91}
\|Y_n-\Psi_n\|\le \epsilon (n-1)C^{n-1}, \quad n\in \{1,\ldots,\nu  \}\,.
\end{equation}
\end{lemma}
\begin{proof}We use induction, again. For $n=1$, the statement holds. Suppose it holds for $n\ge 1$. Denote $E_n=Y_n-\Psi_n$. Then, 
$
E_{n+1}+\Psi_{n+1}=X_{n+1}(E_n+\Psi_n)+\Delta_n
$
and
\[
E_{n+1}=X_{n+1}E_n+\Delta_n, \, \|E_{n+1}\|\le C\|E_n\|+\epsilon\le (n-1)C^n\epsilon+\epsilon\le \epsilon nC^n\,,
\]
as required.
\end{proof}\medskip

\begin{lemma}\label{pron1}
Suppose the numerical sequence $\{x_n\}$ satisfies
\[
|x_{n+1}|\le (1+Ce^{-n})|x_n|+Ce^{-n}, \quad n\ge n_0\,,
\]
then (with a different choice of $C$) we have
\[
|x_n|\le (1+Ce^{-n_0})|x_{n_0}|+Ce^{-n_0},\,\quad n\ge n_0\,.
\]
\end{lemma}
\begin{proof}
Let $\psi_n:=\prod_{j=n_0}^{n} (1+Ce^{-j})$ and $z_n:=\psi_n^{-1}|x_n|$. We have $1\le \psi_n\le 1+Ce^{-n_0}$. Then,
\[
z_{n+1}\le z_n+Ce^{-n}, \quad z_{n_0}\le (1+Ce^{-n_0})|x_{n_0}|\,.
\]
The summation in $n$ yields $z_n\le (1+Ce^{-n_0})|x_{n_0}|+Ce^{-n_0}$. Multiplication with $\psi_n$ gives the required estimate.
\end{proof}

\medskip
The Jost solutions are known to exist for $q\in L^1(\R)$. In the next lemma, we quantify their asymptotics.

\begin{lemma}
Consider $T_+$ that solves \eqref{eq1lol} and assume that $q\in L^1(\R)$. Then, 
\begin{equation}\label{accum}
\|T_+(x,\xi,q)-E(x,\xi)\|\lesssim  \left(\int_x^\infty |q(s)|ds\right)\exp\left(C\int_x^\infty |q(s)|ds\right)\,.
\end{equation}
Suppose $v\in L^1(\R)$. Then, 
\begin{eqnarray}\label{stab1}
\left\{\begin{array}{c}
|a(\xi,q+v)-a(\xi,q)|\lesssim \|v\|_1\exp\left(C(\|q\|_1+\|v\|_1)\right),\\ |b(\xi,q+v)-b(\xi,q)|\lesssim \|v\|_1\exp\left(C(\|q\|_1+\|v\|_1)\right)\,.\end{array}\right.
\end{eqnarray}
\end{lemma}
\begin{proof}
Let $Q:=-i(q\sigma_--\bar{q}\sigma_+)$, $Y:=E^{-1}T_+$ and write
$
i\sigma_3\partial_xT_+=QT_++\frac{\xi}{2}T_+
$
in the form
$
\partial_x Y=\widetilde QY,$ $Y(+\infty,\xi,q)=\idm
$, where $\widetilde Q:=-iE^{-1}\sigma_3QE$. Since $\|\widetilde Q\|\lesssim  |q|$, the analysis of the Volterra integral equation
\begin{equation}\label{lop2}
Y(x,\xi,q)=\idm-\int_x^\infty \widetilde Q(s)Y(s,\xi,q)ds
\end{equation}
gives 
\[
\|Y(x,\xi,q)\|\le 1+C\int_x^\infty |q(s)|\|Y(s,\xi,q)\|ds
\]
after taking the norms on both sides.
Iteration of this bound provides the standard exponential estimate
\begin{equation}\label{sto}
\|Y(x,\xi,q)\|\le \exp\left(C\int_x^\infty |q(s)|ds\right)\,.
\end{equation}
Rewriting \eqref{lop2} and substituting the previous bound, we obtain
\[
\|Y(x,\xi,q)-\idm\|\lesssim  \int_x^\infty |q(s)|\exp\left(C\int_s^\infty |q(t)|dt\right)ds\le\left(\int_x^\infty |q(s)|ds\right)\exp\left(C\int_x^\infty |q(s)|ds\right)\,.
\]
Now, we write
\[
\|T_+-E\|=\|E(Y-\idm)\|=\|(Y-\idm)\|\lesssim \left(\int_x^\infty |q(s)|ds\right)\exp\left(C\int_x^\infty |q(s)|ds\right)\,,
\]
as required. 
Next, we turn to proving \eqref{stab1}. 
From \eqref{fix1} and  \eqref{eq4}, we get 
\begin{equation}\label{cod7}
T^{-1}(\xi,q)=\left(
\begin{array}{cc}
\bar a(\xi,q)&-\bar b(\xi,q)\\
-b(\xi,q)& a(\xi,q)
\end{array}
\right)=\lim_{x\to -\infty} Y(x,\xi,q)\,, \quad \xi\in \R\,,
\end{equation}
where $Y$ was introduced above. For the potential $q+v$, we get
\begin{equation}\label{ein}
\partial_x Y(x,\xi,q+v)=(\widetilde Q+\widetilde V)Y(x,\xi,q+v), \quad Y(+\infty,\xi,q+v)=\idm\,,
\end{equation}
where $\widetilde V:=-iE^{-1}\sigma_3VE$ and $V:=-i(v\sigma_--\bar{v}\sigma_+)$. 
Recall that 
\begin{equation}\label{tswai}
\partial_x Y(x,\xi,q)=\widetilde Q(x)Y(x,\xi,q), \quad Y(+\infty,\xi,q)=\idm\,.
\end{equation}
Subtract one equation from another and write
\begin{equation}\label{sable}
\partial_x U=\widetilde QU+W, \quad U(+\infty)=0\,,
\end{equation}
where $U:=Y(x,\xi,q+v)-Y(x,\xi,q)$ and $W:=\widetilde VY(x,\xi,q+v)$. From \eqref{sto}, we conclude 
\[
\|W(x,\xi)\|\lesssim |v(x)|\exp\bigl(C(\|q\|_1+\|v\|_1)\bigr)\,.
\]
Write \eqref{sable} as an integral equation 
\[
U(x)=-\int_x^\infty \widetilde Q(s)U(s) ds-\int_x^\infty W(s)ds
\]
and take the norms of both sides:
\[
\|U(x)\|\lesssim  \int_x^\infty |q(s)|\|U(s)\|ds+\int_x^\infty \|W(s)\|ds\lesssim  \int_x^\infty |q(s)|\|U(s)\|ds+\|v\|_1(\exp\bigl(C(\|q\|_1+\|v\|_1)\bigr)\,.
\]
Its iteration yields an inequality
\[
\|Y(x,\xi,q+v)-Y(x,\xi,q)\|=
\|U(x)\|\lesssim \|v\|_1(\exp\bigl(C(\|q\|_1+\|v\|_1)\bigr)\,.
\]
Taking the limit as $x\to -\infty$ in the last bound and using \eqref{cod7} gives \eqref{stab1}.
\end{proof}
The next result shows the stability of the fundamental matrix under the  $L^1(\R^+)$ perturbations.\medskip
\begin{lemma}\label{ell1} Suppose $X$ solves \eqref{ks1} where $A\in L^1(\R^+)$. Then, we have
\[
\sup_{x\ge 0, \, k\in \overline{\C^+}}\|X(x,k,A)\|\le  \exp(C\|A\|_1)\,.
\]
\begin{proof}
We can write the Duhamel formula for $X$:
\[
X(x,k,A)= \left(
\begin{array}{cc}
e^{ik x} &0\\
0&1\end{array} \right)-\int_0^x \left(
\begin{array}{cc}
e^{ik (x-\tau)} &0\\
0&1\end{array} \right)\left(
\begin{array}{cc}
0 &\overline{A(\tau)}\\
A(\tau)&0\end{array} \right)X(\tau,k,A)d\tau\,.
\]
Taking the norm of both sides and, recalling that $\Im k\ge 0$, we get
\[
\|X(x,k,A)\|\le 1+C\int_0^x|A(\tau)|\|X(\tau,k,A)\|d\tau\,.
\]
Iterating this inequality, one has 
\[
\|X(x,k,A)\|\le \exp\left(C\int_0^x|A(\tau)|d\tau\right)
\]
and the lemma follows.
\end{proof}
\end{lemma}

\bibliographystyle{plain} 
\bibliography{bibfile}

@article {kov1,
    AUTHOR = {Kova\v{c}, Vjekoslav and Oliveira e Silva, Diogo and Rup\v{c}i\'c, Jelena},
     TITLE = {A sharp nonlinear {H}ausdorff-{Y}oung inequality for small
              potentials},
   JOURNAL = {Proc. Amer. Math. Soc.},
  FJOURNAL = {Proceedings of the American Mathematical Society},
    VOLUME = {147},
      YEAR = {2019},
    NUMBER = {1},
     PAGES = {239--253},
      ISSN = {0002-9939,1088-6826},
   MRCLASS = {42A38 (34L25)},
  MRNUMBER = {3876746},
MRREVIEWER = {Valeri\ S.\ Serov},
       DOI = {10.1090/proc/14268},
       URL = {https://doi-org.ezproxy.library.wisc.edu/10.1090/proc/14268},
}

@article {BD,
    AUTHOR = {Bessonov, Roman V. and Denisov, Sergey A.},
     TITLE = {Sobolev norms of {$L^2$}-solutions to the nonlinear
              {S}chr\"odinger equation},
   JOURNAL = {Pacific J. Math.},
  FJOURNAL = {Pacific Journal of Mathematics},
    VOLUME = {331},
      YEAR = {2024},
    NUMBER = {2},
     PAGES = {217--258},
      ISSN = {0030-8730,1945-5844},
   MRCLASS = {35Q55 (35P25)},
  MRNUMBER = {4817499},
       DOI = {10.2140/pjm.2024.331.217},
       URL = {https://doi-org.ezproxy.library.wisc.edu/10.2140/pjm.2024.331.217},
}

@article {ck,
    AUTHOR = {Christ, M. and Kiselev, A.},
     TITLE = {Scattering and wave operators for one-dimensional
              {S}chr\"odinger operators with slowly decaying nonsmooth
              potentials},
   JOURNAL = {Geom. Funct. Anal.},
  FJOURNAL = {Geometric and Functional Analysis},
    VOLUME = {12},
      YEAR = {2002},
    NUMBER = {6},
     PAGES = {1174--1234},
      ISSN = {1016-443X,1420-8970},
   MRCLASS = {47A40 (34L25 34L40 47B80 47E05 47N50 81U99)},
  MRNUMBER = {1952927},
MRREVIEWER = {Christoph\ M.\ Thiele},
       DOI = {10.1007/s00039-002-1174-9},
       URL = {https://doi-org.ezproxy.library.wisc.edu/10.1007/s00039-002-1174-9},
}

@article {DKr,
    AUTHOR = {Denisov, Sergey A.},
     TITLE = {Continuous analogs of polynomials orthogonal on the unit
              circle and {K}re\u in systems},
   JOURNAL = {IMRS Int. Math. Res. Surv.},
  FJOURNAL = {IMRS. International Mathematics Research Surveys},
      YEAR = {2006},
     PAGES = {Art. ID 54517, 148},
      ISSN = {1687-1308,1687-1324},
   MRCLASS = {42C05 (30C10 34L05 47B36 47E05)},
  MRNUMBER = {2280219},
MRREVIEWER = {Leonid\ B.\ Golinski\u i},
}

@article {Den25,
    AUTHOR = {Denisov, Sergey A.},
     TITLE = {Two quantitative versions of the nonlinear {C}arleson
              conjecture},
   JOURNAL = {C. R. Math. Acad. Sci. Paris},
  FJOURNAL = {Comptes Rendus Math\'ematique. Acad\'emie des Sciences. Paris},
    VOLUME = {363},
      YEAR = {2025},
     PAGES = {1533--1541},
      ISSN = {1631-073X,1778-3569},
   MRCLASS = {42C05 (30H35)},
  MRNUMBER = {5001309},
       DOI = {10.5802/crmath.806},
       URL = {https://doi-org.ezproxy.library.wisc.edu/10.5802/crmath.806},
}

@article {dm,
    AUTHOR = {Denisov, Sergey and Mohamed, Liban},
     TITLE = {Generalizations of {M}enchov-{R}ademacher theorem and
              existence of wave operators in {S}chr\"odinger evolution},
   JOURNAL = {Canad. J. Math.},
  FJOURNAL = {Canadian Journal of Mathematics. Journal Canadien de
              Math\'ematiques},
    VOLUME = {73},
      YEAR = {2021},
    NUMBER = {2},
     PAGES = {360--382},
      ISSN = {0008-414X,1496-4279},
   MRCLASS = {34L25 (34L40 35J10 42C05 47A40)},
  MRNUMBER = {4230378},
MRREVIEWER = {Serge\ C.\ Richard},
       DOI = {10.4153/s0008414x19000646},
       URL = {https://doi-org.ezproxy.library.wisc.edu/10.4153/s0008414x19000646},
}

@book {FD,
    AUTHOR = {Faddeev, Ludwig D. and Takhtajan, Leon A.},
     TITLE = {Hamiltonian methods in the theory of solitons},
    SERIES = {Classics in Mathematics},
     NOTE = {Translated from the 1986 Russian original by Alexey G. Reyman},
 PUBLISHER = {Springer, Berlin},
      YEAR = {2007},
     PAGES = {x+592},
      ISBN = {978-3-540-69843-2},
   MRCLASS = {37K10 (35P25 35Q51 35Q55 35R30 37J35 37N20 81R12)},
  MRNUMBER = {2348643},
}

@article {kls,
    AUTHOR = {Kiselev, Alexander and Last, Yoram and Simon, Barry},
     TITLE = {Modified {P}r\"ufer and {EFGP} transforms and the spectral
              analysis of one-dimensional {S}chr\"odinger operators},
   JOURNAL = {Comm. Math. Phys.},
  FJOURNAL = {Communications in Mathematical Physics},
    VOLUME = {194},
      YEAR = {1998},
    NUMBER = {1},
     PAGES = {1--45},
      ISSN = {0010-3616,1432-0916},
   MRCLASS = {34L40 (39A70 47B39 47B80 47E05 81Q10)},
  MRNUMBER = {1628290},
MRREVIEWER = {G\"unter\ Stolz},
       DOI = {10.1007/s002200050346},
       URL = {https://doi-org.ezproxy.library.wisc.edu/10.1007/s002200050346},
}

@article {kov,
    AUTHOR = {Kova\v{c}, Vjekoslav and Oliveira e Silva, Diogo and Rup\v{c}i\'c, Jelena},
     TITLE = {Asymptotically sharp discrete nonlinear {H}ausdorff-{Y}oung
              inequalities for the {$\rm SU(1,1)$}-valued {F}ourier
              products},
   JOURNAL = {Q. J. Math.},
  FJOURNAL = {The Quarterly Journal of Mathematics},
    VOLUME = {73},
      YEAR = {2022},
    NUMBER = {3},
     PAGES = {1179--1188},
      ISSN = {0033-5606,1464-3847},
   MRCLASS = {42A05 (42C05 43A25)},
  MRNUMBER = {4479843},
       DOI = {10.1093/qmath/haac011},
       URL = {https://doi-org.ezproxy.library.wisc.edu/10.1093/qmath/haac011},
}

@article {musc,
    AUTHOR = {Muscalu, Camil and Tao, Terence and Thiele, Christoph},
     TITLE = {A counterexample to a multilinear endpoint question of
              {C}hrist and {K}iselev},
   JOURNAL = {Math. Res. Lett.},
  FJOURNAL = {Mathematical Research Letters},
    VOLUME = {10},
      YEAR = {2003},
    NUMBER = {2-3},
     PAGES = {237--246},
      ISSN = {1073-2780},
   MRCLASS = {34L40 (47E05)},
  MRNUMBER = {1981900},
MRREVIEWER = {Alexander\ M.\ Gomilko},
       DOI = {10.4310/MRL.2003.v10.n2.a10},
       URL = {https://doi-org.ezproxy.library.wisc.edu/10.4310/MRL.2003.v10.n2.a10},
}

@article {st,
    AUTHOR = {Oberlin, Richard and Seeger, Andreas and Tao, Terence and
              Thiele, Christoph and Wright, James},
     TITLE = {A variation norm {C}arleson theorem},
   JOURNAL = {J. Eur. Math. Soc. (JEMS)},
  FJOURNAL = {Journal of the European Mathematical Society (JEMS)},
    VOLUME = {14},
      YEAR = {2012},
    NUMBER = {2},
     PAGES = {421--464},
      ISSN = {1435-9855,1435-9863},
   MRCLASS = {42A20 (42A16 42A45 42A61)},
  MRNUMBER = {2881301},
MRREVIEWER = {Alexander\ V.\ Tovstolis},
       DOI = {10.4171/JEMS/307},
       URL = {https://doi-org.ezproxy.library.wisc.edu/10.4171/JEMS/307},
}

@article {tt,
    AUTHOR = {Tao, Terence and
              Thiele, Christoph},
     TITLE = {Nonlinear {F}ourier {A}nalysis},
   JOURNAL = {IAS/Park City Graduate Summer School, unpublished lecture notes},
      YEAR = {2012},
       URL = {arxiv:1201.5129},
}
\end{document}